\documentclass[11pt,a4paper]{amsart}
\pdfoutput=1

\usepackage[latin1]{inputenc}
\usepackage[english]{babel}
\usepackage[left=2.75cm,right=2.75cm,bottom=2.75cm,top=2.75cm]{geometry}

\usepackage{amsmath}
\usepackage{amsfonts}
\usepackage{amssymb}
\usepackage{amsrefs}
\usepackage{mathtools}	
\usepackage{bbm}		
\numberwithin{equation}{section}

\usepackage[foot]{amsaddr}

\usepackage{graphicx}
\usepackage[font=small]{caption}
\usepackage{subcaption}
\usepackage{tikz}

\usepackage{xcolor}
\definecolor{thelinkcolor}{RGB}{0,0,150}	
\usepackage[breaklinks=true]{hyperref}
\hypersetup{colorlinks=true,
	urlcolor=thelinkcolor,
	linkcolor=thelinkcolor,
	citecolor=thelinkcolor,
	bookmarksdepth=paragraph,
	pdftitle={Duality of convex relaxations for constrained variational problems},
	pdfauthor={Giovanni Fantuzzi},
	pdfkeywords={Integral functionals, semidefinite programming, convex duality}
}
\usepackage[nameinlink]{cleveref}
\crefname{equation}{}{}
\crefname{theorem}{Theorem}{Theorem}
\crefname{example}{example}{Example}
\crefname{lemma}{Lemma}{Lemma}
\crefname{proposition}{Proposition}{Proposition}
\crefname{figure}{figure}{Figure}
\crefname{table}{table}{Table}

\makeatletter
\newcommand{\newsectionstyle}{%
	\renewcommand{\@secnumfont}{\bfseries}
	\renewcommand\section{\@startsection{section}{2}%
		\z@{0.5\linespacing\@plus0.25\linespacing}{-0.5em}%
		{\hspace{\parindent}\normalfont\bfseries}}%
	\renewcommand\subsection{\@startsection{subsection}{2}%
		\z@{0.5\linespacing\@plus0.25\linespacing}{-0.5em}%
		{\hspace{\parindent}\normalfont\bfseries}}%
}
\let\oldsection\section
\let\old@secnumfont\@secnumfont
\newcommand{\originalsectionstyle}{%
	\let\@secnumfont\old@secnumfont
	\let\section\oldsection
}
\makeatother
\newsectionstyle

\newtheorem{theorem}{Theorem}
\newtheorem{proposition}{Proposition}
\theoremstyle{remark}
\newtheorem{remark}{Remark}

\newcommand{\maxL}{L^*}
\newcommand{\R}{\mathbb{R}}
\newcommand{\dnx}{{\rm d} x}
\newcommand{\dsigma}{{\rm d}\sigma}
\newcommand{\dmu}{{\rm d}\mu}
\newcommand{\dnu}{{\rm d}\nu}
\newcommand{\abs}[1]{\left\vert #1 \right\vert}
\newcommand{\weakstar}{weak-$\ast$}

\DeclareMathOperator{\trace}{tr}

\title[Duality of convex relaxations for constrained variational problems]
{Duality of convex relaxations for\\constrained variational problems}
\date{\today}
\author[G. Fantuzzi]{Giovanni Fantuzzi}
\address{Department of Aeronautics, Imperial College London, SW7 2AZ, London, United Kingdom}
\email{\href{mailto:giovanni.fantuzzi10@imperial.ac.uk}{giovanni.fantuzzi10@imperial.ac.uk}}

\begin{document}
	
\begin{abstract}
	We prove weak duality between two recent convex relaxation methods for bounding the optimal value of a constrained variational problem in which the objective is an integral functional. The first approach, proposed by Valmorbida \textit{et al.}~[\href{https://doi.org/10.1109/TAC.2015.2479135}{IEEE Trans. Automat. Control 61(6):1649--1654, 2016}], replaces the variational problem with a convex program over sufficiently smooth functions, subject to pointwise non-negativity constraints. The second approach, discussed by Korda \textit{et al.} [\href{http://arxiv.org/abs/1804.07565}{arXiv:1804.07565}], relaxes the variational problem into a convex program over scaled probability measures. We also prove that the duality between these infinite-dimensional convex programs is strong, meaning that their optimal values coincide, when the range and gradients of admissible functions in the variational problem are constrained to bounded sets. For variational problems with polynomial data, the optimal values of each convex relaxation can be approximated by solving weakly dual hierarchies of finite-dimensional semidefinite programs (SDPs). These are strongly dual under standard constraint qualification conditions irrespective of whether strong duality holds at the infinite-dimensional level. Thus, the two relaxation approaches are equivalent for the purposes of computations.
\end{abstract}

\maketitle

\section{Introduction}
\label{s:intro}

Constrained variational problems arise in a large number of fields, including nonlinear elasticity, fluid mechanics, and control theory. In this paper, we consider a general class of variational problems relevant to all these application domains: minimize an integral functional,
\begin{equation}
L[u] := \int_X l(x,u,Du) \,\dnx
\end{equation}
where $l \in C(\overline{X}\times\R^n\times\R^{m\times n})$ is given, over all functions $u$ that belong to the set
\begin{align}
\label{e:U-def}
&&&&&&&&\mathcal{U} := \big\{
u \in W^{1,\infty}(X;\,\R^m):\;
\begingroup\textstyle\int_X\endgroup f(x,u,Du)\,\dnx &= 0,&& &&&&&&&&\\[-0.5ex]
&&&&&&&&g(x,u,Du) &= 0 \text{~~~a.e. on~}X,\nonumber\\[-0.5ex]
&&&&&&&&h_1(x,u) &= 0 \text{~~~a.e. on~}\partial X_1,\nonumber\\[-1ex]
&&&&&&&& &\;\;\vdots\nonumber\\[-0.5ex]
&&&&&&&&h_s(x,u) &= 0 \text{~~~a.e. on~}\partial X_s,\nonumber\\[-0.5ex]
&&&&&&&&u \in Y &\subseteq \mathbb{R}^m \text{~~~a.e. on~}X,\nonumber\\[-0.65ex]
&&&&&&&&Du \in Z &\subseteq \mathbb{R}^{m\times n}\text{~a.e. on~}X
\big\}.\nonumber
\end{align}
In these expressions and throughout this work $X \subset \mathbb{R}^n$ is an open bounded domain, whose boundary $\partial X = \overline{X} \setminus X$ is Lipschitz and consists of $s$ smooth surfaces $\partial X_1,\,\ldots,\,\partial X_s$ that are disjoint up to a set of zero surface measure. The Sobolev space $W^{1,\infty}(X;\,\R^m)$ consists of all weakly differentiable and essentially bounded functions $u:X \to \R^m$ whose Jacobian matrix $Du=(\partial_{x_j}u_i)_{i=1,\,\ldots,\,m}^{j=1,\,\ldots,\,n}$ is also essentially bounded. The functions $f,g\in C(\overline{X}\times\R^n\times\R^{m\times n})$ and $h_1 \in C(\partial{X}_1\times\R^n),\,\ldots,\,h_s \in C(\partial{X}_s\times\R^n)$ define integral, differential and boundary constraint on $u$. The sets $Y$ and $Z$, to which the range and gradients of admissible functions are restricted, may or may not be bounded and may coincide with the full spaces $\R^m$ and $\R^{m\times n}$. We assume that $\mathcal{U}$ is nonempty and $L[u]$ is bounded below on $\mathcal{U}$, but not that a minimizer exists. 

Minimizing $L[u]$ over the set $\mathcal{U}$ analytically is often impossible. The minimum value
\begin{equation}
\label{e:L-star}
\maxL := \inf_{u \in \mathcal{U}} L[u]
\end{equation}
is typically approximated numerically either by computing a minimizing sequence using direct optimization techniques, or by discretizing and solving the Euler--Lagrange partial differential equations (PDEs). Unless~\cref{e:L-star} is a convex minimization problem, meaning that $L[u]$ is a convex functional and $\mathcal{U}$ is a convex set, 
such methods return approximate \textit{local} minimizers with no way of checking whether they are globally optimal. In general, therefore, one only obtains an upper bound on $\maxL$. This works investigates complementary techniques that bound $\maxL$ from below.

Two such approaches have been proposed recently, and are particularly interesting because the search for lower bounds on $\maxL$ is posed as a convex optimization problem even when~\cref{e:L-star} is nonconvex. The first approach is to rewrite~\cref{e:L-star} as
\begin{equation}
\label{e:integral-inequality}
\maxL = \sup \left\{b \in \R:\;  \int_{X} l(x,u,Du) \,\dnx \geq b \quad \forall u \in \mathcal{U} \right\},
\end{equation}
augment the integral inequality using Lagrange multipliers to enforce the constraints that define $\mathcal{U}$ and the differential relation between $u$ and $Du$, and replace the augmented inequality with stronger pointwise inequalities on $\overline{X}\times\overline{Y}\times\overline{Z}$ and $\partial{X}\times\overline{Y}$. What results is an infinite-dimensional convex maximization problem  with inequality constraints for $b$ and the Lagrange multipliers. The constraints are polynomial inequalities on semialgebraic sets when the sets $X$, $Y$, $Z$ are semialgebraic, the functions $f$, $g$, $h_1,\,\ldots,\,h_s$, and $l$ are polynomial, and the Lagrange multipliers are restricted to be polynomials of fixed degree. Replacing polynomial non-negativity with sum-of-squares (SOS) conditions enables one to maximize lower bounds on $\maxL$ numerically by solving a hierarchy of semidefinite programs (SDPs), indexed by the degree of the Lagrange multipliers.
This strategy was proposed by Valmorbida \textit{et al.}~\cites{Valmorbida2014,Valmorbida2015a,Valmorbida2015b,Valmorbida2015c} and Ahmadi \textit{et al.}~\cites{Ahmadi2014,Ahmadi2015,Ahmadi2016,Ahmadi2017,Ahmadi2018} in the context of stability analysis, input-output analysis, and safety verification for dynamical systems governed by polynomial PDEs, but applies equally well to constrained variational problems.

The second strategy to bound $\maxL$ from below using convex optimization is to reformulate~\cref{e:L-star} as a minimization problem over the set of occupation and boundary measures on $\mathcal{U}$. These are the images of the Lebesgue measure on $X$ and the surface measure on $\partial{X}$ under the map $x \mapsto (x,u(x),Du(x))$ as $u$ varies in $\mathcal{U}$. The set of such measures is nonconvex in general, but relaxing the minimization to a convex superset gives an infinite-dimensional convex problem whose optimal value is a lower bound on $\maxL$. This, in turn, can be estimated from below using a standard hierarchy of moment-SDP relaxations when the sets $X$, $Y$, $Z$ are semialgebraic and $f,\,g,\,h_1,\,\ldots,\,h_s$ and $l$ are polynomials. These ideas have been applied to analyze and control linear PDEs~\cites{Bertsimas2006,Magron2018}, and have recently been extended to the nonlinear case by Korda \textit{et al.}~\cites{Korda2018}. Similar methods have also been used to approximate solutions to hyperbolic PDEs~\cite{Marx2018}.


This work demonstrates that the infinite-dimensional convex problems obtained with the two relaxation methods just described are weakly dual in the sense of convex duality. In addition, we prove that the duality is strong when $Y$ and $Z$ are bounded sets, so both methods yield the same lower bound on $\maxL$ in this case.
The two hierarchies of finite-dimensional SDPs obtained for variational problems with polynomial data are also weakly dual. 
Strong duality of the SDPs at each level of the hierarchy can be established using standard constraint qualification conditions, which often hold in practice and are independent of whether the infinite-dimensional formulations are strongly dual. Moreover, many popular algorithms for solving SDPs require strong duality to guarantee convergence and avoid poor numerical conditioning. Consequently, the relaxations of~\cref{e:L-star} proposed by~\cites{Valmorbida2014,Valmorbida2015a,Valmorbida2015b,Valmorbida2015c,Ahmadi2014,Ahmadi2015,Ahmadi2016,Ahmadi2017,Ahmadi2018} and~\cite{Korda2018} are equivalent from the point of view of numerical computations.

We present these new results in \cref{s:duality} after reviewing the methods of~\cites{Valmorbida2014,Valmorbida2015a,Valmorbida2015b,Valmorbida2015c,Ahmadi2014,Ahmadi2015,Ahmadi2016,Ahmadi2017,Ahmadi2018} and~\cite{Korda2018} in \cref{s:sos-method,s:measure-method}, respectively. While most of the material contained there has appeared in the cited literature, our discussion slightly differs from previous works and, sometimes, extends them. 
In particular, \cref{s:sos-method} gives a new description of the methods of~\cites{Valmorbida2014,Valmorbida2015a,Valmorbida2015b,Valmorbida2015c,Ahmadi2014,Ahmadi2015,Ahmadi2016,Ahmadi2017,Ahmadi2018}, which does not start with~\cref{e:integral-inequality} and makes the hitherto unrecognized connection with the techniques of~\cite{Korda2018} evident. 
Further comments are offered in \cref{s:conclusion}.

\section{Convex relaxation using Lagrange multipliers}
\label{s:sos-method}

We begin by deriving an infinite-dimensional convex program over Lagrange multipliers, largely based on ideas from~\cites{Valmorbida2014,Valmorbida2015a,Valmorbida2015b,Valmorbida2015c,Ahmadi2014,Ahmadi2015,Ahmadi2016,Ahmadi2017,Ahmadi2018}, whose feasible solutions prove lower bounds on $\maxL$. For notational simplicity we will often write $\Omega := \overline{X} \times \overline{Y} \times \overline{Z}$ and $\Gamma_i := \partial X_i \times \overline{Y}$, where $\partial X_i$ is any of the $s$ smooth components of the boundary $\partial X$. 
We will also write $\abs{X}=\int_X \dnx$ for the volume of $X$ and $\abs{\partial X_i} = \int_{\partial X_i} \dsigma$  for the surface area of each smooth portion of its boundary, where $\dsigma$ is the surface measure.

\subsection{An infinite-dimensional convex program over Lagrange multipliers}
\label{ss:sos-infinite-dimensional}

As in the classical approach to solving~\cref{e:L-star} using the calculus of variations, observe that
\begin{multline}
\label{e:inf-sup-sos}
\maxL = \inf_{\substack{u \in W^{1,\infty}(X;\,\R^m)\\u(x) \in Y \text{ a.e.}\\Du(x) \in Z \text{ a.e.}}}
\sup_{\substack{a\in\R\\\tilde{\phi}\in C(X)\\\tilde{\psi}_i \in C(\partial X_i)}}
\left\{ 
\int_X \left[l(x,u,Du) + a f(x,u,Du) + \tilde{\phi}(x) g(x,u,Du)\right] \dnx \right.\\[-4ex]
\left.  + \sum_{i=1}^s \int_{\partial X_i} \tilde{\psi}_i(x) h_i(x,u) \,\dsigma \right\},
\end{multline}
where $a$, $\tilde{\phi}$ and $\tilde{\psi}_1,\,\ldots,\,\tilde{\psi}_s$ are Lagrange multipliers for the integral, differential and boundary constraints that define $\mathcal{U}$. The spatial structure of the optimal multipliers clearly depends on $u$ and its derivatives. To enforce this without having to prescribe $u$, we let the multipliers be explicit functions of both $u$ and $Du$. Precisely, without loss of generality we consider functions $\phi \in C(\Omega)$ and $\psi_i \in C(\Gamma_i)$, $i=1,\,\ldots,\,s$, and let 
\begin{subequations}
\label{e:phi-psi-x-u-Du}
\begin{gather}
\tilde{\phi}(x) = \phi[x,u(x),Du(x)],\\
\tilde{\psi}_i(x) = \psi_i[x,u(x)], \quad i=1,\,\ldots,\,s.
\end{gather}
\end{subequations}

Another source of difficulty is the differential relation between $u$ and $Du$. To handle this, we introduce a slack function $v \in L^{\infty}(X;\,\R^{m\times n})$ and replace $Du$ by $v$, subject to the constraint $v=Du$. We can then impose the differential relation between $u$ and its derivatives using a special type of vector-valued Lagrange multiplier, without having to consider $Du$ explicitly. The next proposition, which is similar to Lemma 1.1 in~\cite{Korda2018}, makes this precise.

\begin{proposition}
	\label{th:ibp}
	Let $u \in W^{1,\infty}(X;\,\R^m)$ and $v \in L^{\infty}(X;\,\R^{m\times n})$. Then, $v=Du$ a.e. on $X$ if and only if
	\begin{equation}
	\label{e:ibp}
	\int_X \left\{ \nabla_x \cdot \theta(x,u) + \trace[ v\,D_u \theta(x,u) ] \right\} \dnx - \int_{\partial X} \theta(x,u) \cdot \hat{n}(x) \,\dsigma = 0 
	\quad \forall \theta \in C^1(\overline{X} \times \overline{Y};\,\R^n),
	\end{equation}
	where $\nabla_x \cdot \theta(x,u)$ is the divergence of $\theta$ with respect to $x$,
	$D_u\theta = (\partial_{u_j} \theta_i)_{i=1,\,\ldots,\,n}^{j=1,\,\ldots,\,m}$
	is the Jacobian of $\theta$ with respect to $u$,
	$\trace(\cdot)$ is the trace of a square matrix,
	and $\hat{n}$ is the outward unit vector normal to the boundary. 
\end{proposition}
\begin{proof}
	The ``only if" part follows from the divergence theorem. To prove the ``if" part, we proceed as in Lemma 1.1 of~\cite{Korda2018} and set $\theta = (u_j \tau(x) \delta_{ik})_{i=1,\,\ldots,\,n}$ in~\cref{e:ibp} for fixed $(j,k)\in\{1,\,\ldots,\,m\}\times\{1,\,\ldots,\,n\}$ and arbitrary $\tau \in C^\infty(\overline{X})$, where $\delta_{ik}$ is the usual Kronecker delta.  Upon integrating by parts the term $\nabla_x\cdot \theta = u_j\partial_{x_k} \tau$  we obtain $\int_X (v_{jk} - \partial_{x_k}u_j) \tau\, \dnx = 0$, which implies $v_{jk}=\partial_{x_k}u_j$ a.e. because $\tau$ is arbitrary. Repeating this argument for all pairs $(j,k)$ concludes the proof.
\end{proof}

Combining \Cref{th:ibp} with~\cref{e:inf-sup-sos} and~(\ref{e:phi-psi-x-u-Du}a,b) shows that
\begin{equation}
\label{e:inf-sup-sos-3}
\maxL =
\inf_{\substack{u \in W^{1,\infty}(X;\,\R^m)\\u(x) \in Y \text{ a.e.}\\v(x) \in Z \text{ a.e.}}}
\sup_{\substack{\alpha\in\R\\\phi\in C(\Omega)\\\psi_i \in C(\Gamma_i)\\\theta \in C^1(\overline{X} \times \overline{Y};\,\R^n)}}
\left\{ \int_X F(x,u,v) \,\dnx + \sum_{i=1}^{s} \int_{\partial X_i} G_i(x,u) \,\dsigma \right\},
\end{equation}
where
\begin{subequations}
	\label{e:FG}
	\begin{align}
	\label{e:F-def}
	F(x,u,v) :=\;&l(x,u,v) + a f(x,u,v) \\&+ \phi(x,u,v) g(x,u,v) + \nabla_x \cdot \theta(x,u) +  \trace\!\left[ v\,D_u \theta(x,u) \right],
	\nonumber\\[1ex]
	G_i(x,u) := \,& \psi_i(x,u) h_i(x,u) - \theta(x,u) \cdot \hat{n}(x).
	\end{align}
\end{subequations}

Solving~\eqref{e:inf-sup-sos-3} is clearly just as hard as solving the original variational problem~\eqref{e:L-star}. However, starting with~\eqref{e:inf-sup-sos-3} it is almost immediate to derive an infinite-dimensional convex program that proves a lower bound on $\maxL$. First, we exchange the inf and sup at the expense of replacing equality with a lower bound. Second, we estimate the integrals of $F$ and $G_i$ in an elementary way using the constraints $u(x) \in Y$, $v(x) \in Z$ to obtain
\begin{equation}
\label{e:sos-final}
\maxL \geq \sup_{\substack{a\in\R\\\phi\in C(\Omega)\\\psi_i \in C(\Gamma_i)\\\theta \in C^1(\overline{X} \times \overline{Y};\,\R^n)}}
\left\{ 
	\abs{X}\inf_{(x,y,z)\in \Omega} F(x,y,z)  
  + \sum_{i=1}^s \abs{\partial X_i} \inf_{(x,y)\in \Gamma_i} G_i(x,y)
\right\}
=: \mathbb{D}.
\end{equation}
This is a convex program for $a,\,\phi,\,\psi_1,\,\ldots,\,\psi_s$ and $\theta$ because $F$ and each $G_i$ depend affinely on them, so 
$(\alpha,\phi,\psi_1,\,\ldots,\,\psi_s,\theta)\mapsto \inf_{\Omega} F(x,y,z)$ and $(\alpha,\phi,\psi_1,\,\ldots,\,\psi_s,\theta)\mapsto \inf_{\Gamma_i} G_i(x,y)$ are convex functions. Convexity can be made more explicit by introducing slack variables $b$ and $c_1,\,\ldots,\,c_s$ and rewriting
\vspace{-3ex}
\begin{equation}
\label{e:sos-final-with-slacks}
\mathbb{D} = 
\sup_{\substack{a,b,c_i\in\R\\\phi\in C(\Omega)\\\psi_i \in C(\Gamma_i)\\\theta \in C^1(\overline{X} \times \overline{Y};\,\R^n)}}
\!\!\!\!\!\!\begin{aligned}
\\
\bigg\{
\abs{X} b + \sum_{i=1}^s \abs{\partial X_i} c_i:\quad
F(x,y,z) - b &\geq 0 \text{ on } \Omega,\\[-3ex]
G_i(x,y) - c_i &\geq 0 \text{ on } \Gamma_i,\quad i = 1,\,\ldots,\,s
\bigg\}.
\end{aligned}
\end{equation}

The maximization problem on the righthand side is still hard to solve even with computer assistance. However, observe that \textit{any} choice of $a,\,\phi,\,\psi_1,\,\ldots,\,\psi_s$ and $\theta$ such that $F(x,y,z)$ and $G_1(x,y),\,\ldots,\,G_s(x,y)$ are bounded below on $\Omega$ and $\Gamma_1,\,\ldots,\,\Gamma_s$ produces a lower bound on $\maxL$. This makes it possible to prove suboptimal bounds analytically. Moreover, as we discuss next, in certain cases it is possible to optimize $a,\,\phi,\,\tilde{\psi}_1,\,\ldots,\,\tilde{\psi}_s$ and $\theta$ numerically.

\begin{remark}
	To pass from~\cref{e:inf-sup-sos-3} to~\cref{e:sos-final} we have estimated $\int_{\partial X} G_i(x,u) \dsigma \geq \abs{\partial X_i} \inf_{\Gamma_i} G_i(x,y)$ for each $i=1,\,\ldots,\,s$. As already observed in~\cite{Valmorbida2015b}, we could improve these generic estimates by optimizing a lower bound on each $\int_{\partial X_i} G_i(x,u) \dsigma$ while optimizing $a$, $\phi$, $\psi_i$ and $\theta$. Indeed, minimizing $\int_{\partial X_i} G_i(x,u) \dsigma$ is a variational problem with affine dependence on $a$, $\phi$, $\psi_i$ and $\theta$ on an $(n-1)$-dimensional surface. Consequently, it can be relaxed into a convex program exactly as explained in this section if an explicit $(n-1)$-dimensional parametrization of $\partial X_i$ is available. This, in turn, requires estimates on $(n-2)$-dimensional integrals, and the procedure can be iterated until one is left with a one-dimensional problem. We do not pursue this approach here, but~\cite{Valmorbida2015b} gives a detailed discussion for square domains in $\R^2$.
\end{remark}

\subsection{Optimizing bounds by solving SDPs}
\label{ss:sos-sdp}

Let us now restrict the attention to variational problems with polynomial data. Specifically, we assume that the functions $f$, $g$, $h_i$ and $l$ are polynomials and that the sets $X$, $Y$ and $Z$ are semialgebraic, i.e., they are defined by a finite number of polynomial equations and inequalities. This means that we can find polynomials $p_1,\,\ldots,\,p_r$ such that
\begin{equation}
\label{e:Omega-semialgebraic}
\Omega = \overline{X} \times \overline{Y} \times \overline{Z} = \{(x,y,z):\, p_1(x,y,z)\geq 0,\,\ldots,\,p_r(x,y,z)\geq 0 \}.
\end{equation}
Similarly, for each smooth portion $\partial X_i$ of the boundary there exist polynomials $q_{i,1},\,\ldots,\,q_{i,t_i}$ such that 
\begin{equation}
\label{e:Gamma-semialgebraic}
\Gamma_i := \partial{X}_i \times \overline{Y} = \{(x,y):\, q_{i,1}(x,y)\geq 0,\,\ldots,\,q_{i,t_i}(x,y)\geq 0 \}.
\end{equation}
For simplicity, we assume that the outward unit vector $\hat{n}$ normal to each $\partial X_i$ is polynomial. This is true, for example, when $X$ is a polyhedral domain. 
\Cref{app:nonpolynomial-normal} shows that some cases in which $\hat{n}$ is not polynomial  can also be handled after a small modification of~\cref{e:sos-final-with-slacks}.

Given an integer $d$, let us restrict the optimization in~\cref{e:sos-final-with-slacks} to degree-$d$ polynomials $\phi \in \R_{d}[x,y,z]$ and $\,\psi_1,\,\ldots,\,\psi_s,\,\theta_1,\,\ldots,\,\theta_n \in \R_{d}[x,y]$. Then, the constraints are polynomial inequalities on semialgebraic sets and depend affinely on $a$, $b$, $c_1,\,\ldots,\,c_s$ and the (finitely many) coefficients of $\phi$, $\psi_1,\,\ldots,\,\psi_s$ and $\theta = (\theta_1,\,\ldots,\,\theta_n)$. These are NP-hard in general, but can be strengthened into tractable conditions by requiring that non-negative polynomials are representable as weighted sums of squares. More precisely, to the semialgebraic sets $\Omega$ and $\Gamma_1,\ldots\Gamma_s$ we associate the sets of polynomials
\begin{subequations}
	\begin{gather}
	Q(\Omega) := \{ w \in \R[x,y,z] :\; w = \sigma_0 + \sigma_1 p_1 + \cdots + \sigma_r p_r,\quad \sigma_0,\,\ldots,\,\sigma_r \in \Sigma[x,y,z]  \},\\
	Q(\Gamma_i) := \{ w \in \R[x,y] :\; w = \sigma_0 + \sigma_1 q_{i,1} + \cdots + \sigma_s q_{i,t_i},\quad \sigma_0,\,\ldots,\,\sigma_{t_i} \in \Sigma[x,y]  \},
	\end{gather}
\end{subequations}
where $\R[x,y,z]$ (resp. $\R[x,y]$) is the space of polynomials in variables $x,y,z$ (resp. $x,y$) and $\Sigma[x,y,z]$ (resp. $\Sigma[x,y]$) is its subset of SOS polynomials. In other words, elements of $Q(\Omega)$ are weighted sums of $r+1$ SOS polynomials with weights $1,\,p_1,\,\ldots,\,p_r$, and similarly for each $Q(\Gamma_i)$. All polynomials in $Q(\Omega)$ and $Q(\Gamma_i)$ are clearly non-negative on $\Omega$ and $\Gamma_i$, respectively, although the converse is not true in general. Then, we can replace the polynomial inequalities in~\cref{e:sos-final-with-slacks} with weighted SOS constraints to obtain
%
\vspace{-3ex}
\begin{equation}
\label{e:sos-with-slacks}
\mathbb{D} \geq \sup_{\substack{a,b,c_i \in \R\\\phi \in \R_{d}[x,y,z]\\\psi_i\in \R_{d}[x,y]\\\theta_1,\ldots,\theta_n \in \R_{d}[x,y]}} \;
\begin{aligned}
\\
\big\{
\abs{X} b + \abs{\partial X} c:\quad 
F(x,y,z) - b &\in Q(\Omega),\\[-0.25ex]
G_i(x,y) - c &\in Q(\Gamma_i),\quad i=1,\,\ldots,\,s
\big\}.
\end{aligned}
\end{equation}
It is well known that optimization problems with weighted SOS constraints can be recast into SDPs (see, e.g., section 2.4.2 in~\cite{Lasserre2015}), and can therefore be solved using a variety of algorithms with polynomial-time complexity. In addition, while the bounds obtained with~\eqref{e:sos-with-slacks} for finite $d$ are typically strictly lower than $\mathbb{D}$, they converge to $\mathbb{D}$ as $d$ is raised provided that $\overline{X}$, $\overline{Y}$ and $\overline{Z}$ satisfy suitable compactness assumptions. The next proposition makes this statement precise and---as the results in \cref{s:duality} imply---is the dual counterpart to Theorem 3 in~\cite{Korda2018} on the convergence of the moment-SDP relaxations described in the next section.

\begin{proposition}
	Suppose that $\overline{X}$, $\overline{Y}$ and $\overline{Z}$ are compact. Suppose also that there exist positive constants $C_0,\,\ldots,\,C_s$ such that $C_0 - \|x\|^2-\|y\|^2-\|z\|^2$ is in $Q(\Omega)$ and $C_i - \|x\|^2-\|y\|^2$ is in $Q(\Gamma_i)$ for each $i=1,\,\ldots,\,s$. Then, the righthand side of~\cref{e:sos-with-slacks} converges to $\mathbb{D}$ as $d \to \infty$.
\end{proposition}

\begin{proof}
The proof follows a standard template in SOS optimization, which combines polynomial approximation of $C^1$ functions and their derivatives with Putinar's Positivstellensatz~\cite[Lemma~4.1]{Putinar1993} on the existence of weighted SOS representations for strictly positive polynomials on a class of compact semialgeraic sets that includes $\Omega$ and $\Gamma_1,\,\ldots,\,\Gamma_s$ (see section 2.4.1 of~\cite{Lasserre2015} for more on this result). The details are left to the interested reader.
\end{proof}

\section{Convex relaxation using measures}
\label{s:measure-method}

Korda \textit{et al.}~\cite{Korda2018} proposed a different approach to bounding $\maxL$ from below. The key idea is to relax a variational problem into an infinite-dimensional convex program over scaled probability measures. For variational problems with polynomial data, the optimization over measures can be replaced with SDPs that optimize finite sequences of their moments. Here we review this approach in the context of~\cref{e:L-star}.

\subsection{A convex program over non-negative measures}
\label{ss:measures-infinite-dimensional}

For each $u$ in the set $\mathcal{U}$ of admissible functions for~\cref{e:L-star}, consider the maps
\begin{equation}
\begin{aligned}
\zeta^u: \;X &\to \overline{X} \times \overline{Y} \times \overline{Z}\\
x  &\mapsto (x,u(x),Du(x)),
\end{aligned}
\qquad\qquad
\begin{aligned}
\chi_i^u: \,\partial X_i &\to \partial X_i \times \overline{Y}\\
x  &\mapsto (x,u(x)).
\end{aligned}
\end{equation}
Let $\mu^u:= \zeta^u\sharp\dnx$ be the pushforward by $\zeta^u$ of the Lebesgue measure on $X$. Similarly, for each $i=1,\,\ldots,\,s$ let $\nu_i^u:=\chi_i^u\sharp\dsigma$  be the pushforward by $\chi_i^u$ of the surface measure on $\partial X_i$. Following~\cite{Korda2018}, we refer to $\mu^u$ and $\nu_1^u,\,\ldots,\,\nu_s^u$ as the occupation and boundary measures of $u$. They are defined on $\Omega$ and $\Gamma_1,\,\ldots,\,\Gamma_s$, respectively, and satisfy
\begin{subequations}
	\begin{gather}
	\label{e:mu}
	\int_X \eta(x,u,Du) \,\dnx = \iiint_{\Omega} \eta(x,y,z) \,\dmu_u(x,y,z) =: \langle\eta,\mu^u\rangle,\\
	\label{e:nu}
	\int_{\partial X_i} \xi(x,u) \,\dsigma = \iint_{\Gamma_i} \xi(x,y) \,\dnu_i^u(x,y) =: \langle\xi,\nu_i^u\rangle,
	\end{gather}
\end{subequations}
whenever $\eta$ and $\xi$ are such that the lefthand sides are well defined (see, e.g.,~\cite[Theorem~3.6.1]{Bogachev2007}). In particular, for each $u \in \mathcal{U}$ we have $L[u] = \langle l,\mu^u\rangle$ and we can rewrite~\cref{e:L-star} as a minimization problem over occupation measures:
\begin{equation}
\label{e:measure-vanilla}
\maxL = \inf_{\mu^u} \langle l,\mu^u\rangle.
\end{equation}

While the objective in this problem is linear in $\mu^u$, the set of occupation measures is generally not convex and~\cref{e:measure-vanilla} is no easier than~\cref{e:L-star}.
The strategy of Korda \textit{et al.}~\cite{Korda2018} is to construct a convex set $\mathcal{M}$ of measures $(\mu,\nu_1,\ldots,\nu_s)$ that contains all occupation and boundary measures, and minimize $\langle l,\mu \rangle$ as a linear function over $\mathcal{M}$. Clearly, this is a convex program that yields a lower bound on $\maxL$.

To construct $\mathcal{M}$, note that all occupation measures $\mu^u$ and boundary measures $\nu_1^u,\,\ldots,\,\nu_s^u$ are non-negative, which we write as $\mu^u,\nu_1^u,\,\ldots,\,\nu_s^u \geq 0$. This follows immediately after fixing $\eta$ and $\xi$ in~\cref{e:mu,e:nu} to be the characteristic functions of any $E\subseteq X$ and $E\subseteq \partial X_i$, respectively. In particular, for $E=X$ and $E=\partial X_i$ we obtain
\begin{subequations}
	\label{e:measure-norms}
	\begin{gather}
	\|\mu^u\| := \langle 1,\mu^u\rangle= \abs{X},\\
	\|\nu_i^u\| := \langle 1,\nu_i^u\rangle= \abs{\partial X_i}.
	\end{gather}
\end{subequations}
This shows that occupation measures are scaled probability measures on $\Omega$ with mass $\abs{X}$, while each boundary measure $\nu_i^u$ is a scaled probability measure on $\Gamma_i$ with mass~$\abs{\partial{X}_i}$.

Additional conditions on the occupation and boundary measures can be derived from the integral, differential and boundary constraints that define the set $\mathcal{U}$ of admissible functions for~\cref{e:L-star}. Specifically, applying~\cref{e:mu} to the integral constraint $\int_X f(x,u,Du) \,\dnx = 0$ gives
\begin{equation}
\label{e:measure-integral-constraints}
\langle f, \mu^u \rangle = 0.
\end{equation}
%
In addition, multiplying the PDE constraint $g(x,u,Du)=0$ by $\phi(x,u,Du)$ for any $\phi \in C(\Omega)$, integrating the results over $X$, and using~\cref{e:mu} we obtain
\begin{equation}
\label{e:measure-differential-constraints}
\langle \phi g, \mu^u \rangle = 0 \quad \forall \phi \in C(\Omega).
\end{equation}
Similarly, for each $i=1,\,\ldots,\,s$, multiplying the boundary constraint $h_i(x,u)=0$ by $\psi_i(x,u)$ for any $\psi_i \in C(\Gamma_i)$, integrating over $\partial X_i$, and using~\cref{e:nu} we conclude that
\begin{equation}
\langle \psi_i h_i, \nu_i^u \rangle = 0 \quad \forall \psi_i \in C(\Gamma_i), \; i=1,\,\ldots,\,s.
\label{e:measure-boundary-constraints}
\end{equation}
The uncountably infinite sets of conditions~\cref{e:measure-differential-constraints,e:measure-boundary-constraints} correspond to equations (14b,c) in~\cite{Korda2018}. In fact,~\cref{e:measure-boundary-constraints} slightly generalizes (14c) in~\cite{Korda2018} because we do not sum over $i$, and we do not require that $\psi_i$ and $\psi_j$ match on $\partial X_i \cap \partial X_j$ whenever $\partial X_i$ is adjacent to $\partial X_j$.

The last set of constraints on occupation and boundary measures considered in~\cite{Korda2018} comes from the divergence theorem. Specifically, applying~\cref{e:mu,e:nu} to~\cref{e:ibp} with $v=Du$ after writing the boundary integral as the sum of integrals over $\partial X_1,\,\ldots,\,\partial X_s$ gives, with notation analogous to \cref{th:ibp},
\begin{equation}
\label{e:measure-ibp}
\langle \nabla_x \cdot \theta + \trace(z\,D_y \theta), \mu^u \rangle - 
\sum_{i=1}^{s}\langle \theta \cdot \hat{n}, \nu_i^u \rangle = 0 \quad \forall \theta \in C^1(\overline{X} \times \overline{Y};\,\R^n).
\end{equation}

Combining all the above results we conclude that the occupation and boundary measures associated to feasible $u$ for~\cref{e:L-star} belong to the set
\begin{equation}
\mathcal{M} := \{(\mu,\nu_1,\ldots,\nu_s):\;\mu,\nu_1,\ldots,\nu_s \geq 0 \text{ and satisfy (\ref{e:measure-norms}a,b)--\cref{e:measure-ibp}} \}.
\end{equation}
It is easy to verify that $\mathcal{M}$ is convex, so in general not all its elements are tuples of occupation and boundary measures.
Thus, minimizing $\langle l,\mu \rangle$ over $\mathcal{M}$ typically yields a lower bound on~$\maxL$:
\begin{equation}
\label{e:moments-final}
\maxL \geq \inf_{(\mu,\nu_1,\ldots,\nu_s) \in \mathcal{M}} \langle l,\mu \rangle =: \mathbb{P}.
\end{equation}
The minimization problem on the righthand side, which is a restatement of problem (16) from~\cite{Korda2018} in the present context, is an infinite-dimensional convex program. As for the lower bound in~\cref{e:sos-final-with-slacks}, solving this convex program is generally beyond the reach of analytical work. However, in stark constrast to~\cref{e:sos-final-with-slacks}, it does \textit{not} suffice to find a feasible point for~\cref{e:moments-final} because only the optimal value $\mathbb{P}$ is guaranteed to be a lower bound on $\maxL$. Thus, one must either construct a minimizer (or minimizing sequence), or estimate $\mathbb{P}$ from below.

\begin{remark}
	As already discussed in~\cite{Korda2018}, it is of great interest to determine conditions on the functions $f$, $g$, $h_i$ and the sets $X$, $Y$, $Z$ under which the set $\mathcal{M}$ is the closed convex hull of the set of occupation and boundary measures in the \weakstar\ topology. If this were the case, the inequality in~\cref{e:moments-final} would in fact be an equality, and the convex relaxation described in this section would be tight. We will not consider this problem in this work.
\end{remark}

\subsection{Optimizing bounds using SDPs}
\label{ss:measure-sdp}

As in section \cref{ss:sos-sdp}, let us now restrict the attention to polynomial $f$, $g$, $h_i$ and $l$, and to semialgebraic sets $X$, $Y$ and $Z$. 
In this case, the measures $\mu,\nu_1,\ldots,\nu_s$ are supported on semialgebraic sets and $\mathbb{P}$ can be estimated from below by solving finite-dimensional SDPs derived with the so-called moment-SDP relaxation procedure described in~\cite{Korda2018}. We only give a brief overview of this approach here, and refer the interested reader to section 4 in~\cite{Korda2018} for a full discussion.

Moment-SDP relaxations rely on two observations. The first is that one can relax the minimization in~\cref{e:moments-final} by imposing~\cref{e:measure-integral-constraints,e:measure-differential-constraints,e:measure-boundary-constraints,e:measure-ibp} only over finitely many choices for $\phi$, $\psi_i$ and $\theta=(\theta_1,\,\ldots,\,\theta_n)$. The resulting problem is still infinite-dimensional, because the optimization variables are non-negative measures, but has a finite number of constraints.
The second observation is that, when $f,\,g,\,h_1,\,\ldots,\,h_s$ and $l$ are polynomials, the objective $\langle l, \mu \rangle$ in~\cref{e:moments-final} is a finite linear combination of moments of $\mu$. Similarly,~\cref{e:measure-integral-constraints} and~(\ref{e:measure-norms}a,b) are linear equalities relating a finite number of moments of $\mu$ and~$\nu_1,\,\ldots,\,\nu_s$. More constraints on the moments can be obtained from~\cref{e:measure-differential-constraints}, \cref{e:measure-boundary-constraints} and~\cref{e:measure-ibp} by taking
\begin{subequations}
\label{e:test-moments}
\begin{align}
\phi &\in \{x^\alpha y^\beta z^\gamma\}_{\abs{\alpha}+\abs{\beta}+\abs{\gamma} \leq d}
\\
\psi_1,\,\ldots,\,\psi_s  &\in \{x^\alpha y^\beta\}_{\abs{\alpha}+\abs{\beta}\leq d},
\\
\theta_1,\,\ldots,\,\theta_n  &\in \{x^\alpha y^\beta\}_{\abs{\alpha}+\abs{\beta}\leq d}
\end{align}
\end{subequations}
to be monomials of total degree no larger than some integer $d$ of choice. (In these expressions we have used standard multi-index notation, e.g., $x^\alpha = x_1^{\alpha_1}\cdots x_n^{\alpha_x}$ and $\abs{\alpha} = \alpha_1 + \cdots +\alpha_n$.) Thus, the minimization in~\cref{e:moments-final} can be relaxed into a finite-dimensional minimization problem for finitely many moments of $\mu$ and $\nu_1,\,\ldots,\,\nu_s$. This, in turn, can be relaxed into an SDP because sets of truncated sequences of moments of measures with semialgebraic supports admit semidefinite-representable outer approximations; see, for instance,~\cite[section 3.2]{Korda2018} and the monographs~\cite{Laurent2009,Lasserre2015}.

\section{Duality}
\label{s:duality}

The functions $\phi$, $\psi_1,\,\ldots,\,\psi_s$ and $\theta$ used in~\cref{s:sos-method,s:measure-method} are evidently very similar. This similarity is not an artefact of our notation, but stems from the fact that the convex program over Lagrange multipliers on the righthand side of~\cref{e:sos-final} (equivalently,~\cref{e:sos-final-with-slacks}) is the Lagrangian dual of the convex program over measures on the righthand side of~\cref{e:moments-final}.  This observation is made precise by the following result, which is our main contribution.

\begin{theorem}
	\label{th:weak-duality}	
	The convex programs on the righthand sides of~\cref{e:sos-final,e:moments-final} are weakly dual and $\mathbb{D} \leq \mathbb{P}$.
	The duality is strong, meaning that $\mathbb{P} = \mathbb{D}$,
	if the sets $Y$ and $Z$ are bounded.
\end{theorem}

\begin{proof}
	To prove weak duality, recall the definitions of $F$ and $G$ from~(\ref{e:FG}a,b) and note that
	\vspace*{-35pt}
	\begin{align}
	\label{e:duality-proof}
		\inf_{(\mu,\nu_1,\ldots,\nu_s) \in \mathcal{M}} \langle l,\mu \rangle 
		&=
		\inf_{\substack{\mu,\nu_1,\ldots,\nu_s\geq 0\\\text{s.t. (\ref{e:measure-norms}a,b)}}}
		\sup_{\substack{a\in\R\\\phi\in C(\Omega)\\\psi_i \in C(\Gamma_i)\\\theta \in C^1(\overline{X} \times \overline{Y};\,\R^n)}}\!\!\!\!\!\!
		\begin{aligned}
		\\[25pt]
		\bigg\{ \langle \underbrace{l + af + \phi g + \nabla_x \cdot \theta + \trace(z\,D_y \theta)}_{F}, &\mu \rangle
		\\[-1.25ex]
		+ \sum_{i=1}^s \langle \underbrace{\psi_i h_i - \theta \cdot \hat{n}}_{G_i}&, \nu_i \rangle
		\bigg\}
		\end{aligned}
		\\
		& \geq
		\sup_{\substack{a\in\R\\\phi\in C(\Omega)\\\psi_i \in C(\Gamma_i)\\\theta \in C^1(\overline{X} \times \overline{Y};\,\R^n)}}
		\inf_{\substack{\mu,\nu_1,\ldots,\nu_s\geq 0\\\text{s.t. (\ref{e:measure-norms}a,b)}}}
		\bigg\{
		\langle F, \mu \rangle 
		+ \sum_{i=1}^s \langle G_i, \nu_i \rangle
		\bigg\}
		\nonumber \\
		&=
		\sup_{\substack{a\in\R\\\phi\in C(\Omega)\\\psi_i \in C(\Gamma_i)\\\theta \in C^1(\overline{X} \times \overline{Y};\,\R^n)}}
		\!\!\!\!\!\!\bigg\{ 
		\abs{X}\inf_{(x,y,z)\in \Omega} F(x,y,z)  
		+ \sum_{i=1}^s\abs{\partial X_i} \inf_{(x,y)\in \Gamma_i} G_i(x,y)
		\bigg\}.
		\nonumber
	\end{align}
	The last equality follows after observing that if $\{(x_j,y_j,z_j)\}_{j\geq 1} \subset \Omega$ and $\{(x_{ij},y_{ij})\}_{j\geq 1} \subset \Gamma_i$ are minimizing sequences for $F$ and $G_i$, then the collections of scaled Dirac measures 
	\begin{equation}
	(\mu,\nu_1,\ldots,\nu_s)_j := \left( 
		\abs{X}\delta_{(x_j,y_j,z_j)},\,
		\abs{\partial X_1}\delta_{(x_{1j},y_{1j})},\,
		\ldots,\,
		\abs{\partial X_s}\delta_{(x_{sj},y_{sj})}
	\right)
	\end{equation}
	form a minimizing sequence for the inner infimum on the second line.
	
	To prove strong duality when $Y$ and $Z$ are bounded observe that $\overline{Y}$ and $\overline{Z}$, hence $\Omega=\overline{X}\times\overline{Y}\times\overline{Z}$ and $\Gamma_i=\partial{X}_i\times\overline{Y}$, are compact. Then, an abstract minimax theorem due to Sion~\cite[Theorem 3.3]{Sion1958} guarantees that equality is preserved when exchanging the inf and the sup in the second line of~\cref{e:duality-proof}. To verify that the hypotheses of Sion's theorem hold in our case, observe that the tuple $(a,\phi,\psi_1,\ldots,\psi_s,\theta)$ belongs to the product space 
	\begin{equation}
	M := \R\times C(\Omega) \times C(\Gamma_1) \times \cdots \times C(\Gamma_s) \times C^1(\overline{X} \times \overline{Y};\,\R^n).
	\end{equation}
	We consider $M$ as a linear (hence, convex) space with the product topology generated by the usual norm topologies on $\R$, $C(\Omega)$, $C(\Gamma_i)$ and $C^1(\overline{X} \times \overline{Y};\,\R^n)$. 
	Moreover, the space
	\begin{equation}
	N=\{(\mu,\nu_1,\ldots,\nu_s):\;\mu,\nu_1,\ldots,\nu_s \geq 0 \text{ subject to (\ref{e:measure-norms}a,b)}\},
	\end{equation}
	is the product of spaces of scaled probability measures on the compact sets $\Omega,\,\Gamma_1,\,\ldots,\,\Gamma_s$ with mass $\abs{X},\,\abs{\partial X_1},\,\ldots,\,\abs{\partial X_s}$, respectively. We consider $N$ in the product \weakstar\ topology, so it is a compact linear (hence, convex) space. Finally, the function
	\begin{equation}
	(a,\phi,\psi_1,\ldots,\psi_s,\theta) \mapsto \langle F, \mu \rangle + \sum_{i=1}^s\langle G_i, \nu_i \rangle
	\end{equation}
	is linear and continuous (hence, quasiconcave and upper semicontinuous) on $M$ for each $(\mu,\nu_1,\ldots,\nu_s) \in N$. Conversely, for each $(a,\phi,\psi_1,\ldots,\psi_s,\theta) \in M$ the function
	\begin{equation}
	(\mu,\nu_1,\ldots,\nu_s) \mapsto \langle F, \mu \rangle + \sum_{i=1}^s\langle G_i, \nu_i \rangle
	\end{equation}
	is linear and continuous (hence, quasiconvex and lower semicontinuous) on $N$.
\end{proof}

In light of the well known duality between the cones of weighted SOS polynomials and moment sequences~\cite{Laurent2009,Lasserre2015}, \cref{th:weak-duality} implies that the finite-dimensional SDP relaxations briefly described in \cref{ss:sos-sdp,ss:measure-sdp} are also weakly dual when the same value of $d$ is taken in~\cref{e:sos-with-slacks} and~(\ref{e:test-moments}a,b,c). Strong duality at the level of SDPs holds under general constraint qualification conditions, such as Slater's condition~\cite[section 5.9.2]{Boyd2004}, which are often satisfied in practice and can be verified independently of whether $\mathbb{P}=\mathbb{D}$. Moreover, strong duality is needed to guarantee convergence and good numerical performance of many primal-dual algorithms for semidefinite programming that solve the Karush--Kuhn--Tucker (KKT) optimality conditions. Since many commonly used SDP solvers implement such algorithms, from the point of view of numerical computations the two approaches to bounding $\maxL$ described in \cref{s:sos-method,s:measure-method} are equivalent.

\section{Further comments}
\label{s:conclusion}

In this work we have demonstrated the duality between the two convex relaxation methods proposed by~\cites{Valmorbida2014,Valmorbida2015a,Valmorbida2015b,Valmorbida2015c,Ahmadi2014,Ahmadi2015,Ahmadi2016,Ahmadi2017,Ahmadi2018} and~\cite{Korda2018} to bound from below the optimal value $\maxL$ of a constrained variational problem. Precisely, the convex program over scaled probability measures formulated in~\cite{Korda2018} is weakly dual to the infinite-dimensional convex program over Lagrange multipliers that we have derived using ideas from~\cites{Valmorbida2014,Valmorbida2015a,Valmorbida2015b,Valmorbida2015c,Ahmadi2014,Ahmadi2015,Ahmadi2016,Ahmadi2017,Ahmadi2018}. Furthermore, we have proven that the duality is strong when the range and gradients of admissible function in the original variational problem are constrained to bounded sets. 

For problems with polynomial data, the hierarchies of SDPs obtained with the SOS and moment-SDP relaxations briefly described in \cref{ss:sos-sdp,ss:measure-sdp} are also dual. Thus, they can be interpreted as extensions to polynomial variational problems of the well known dual hierarchies of SOS and moment-SDP relaxations for standard polynomial optimization problems. In contrast to standard polynomial optimization problems, however, there is currently no guarantee that the lower bounds on $\maxL$ computed using these hierarchies can be made arbitrarily sharp. Numerical experiments from~\cites{Valmorbida2014,Valmorbida2015a,Valmorbida2015b,Valmorbida2015c,Ahmadi2014,Ahmadi2015,Ahmadi2016,Ahmadi2017,Ahmadi2018} and~\cite{Korda2018} suggest that sharp bounds are sometimes possible, but the issue should be investigated in more depth both theoretically and computationally.

All our results can be extended to variational problems with more general choices for the set $\mathcal{U}$ in~\cref{e:U-def}. For instance, the extension to problems with multiple PDE, boundary, and integral constraints is immediate. In addition, one can replace the Sobolev space $W^{1,\infty}(X;\,\R^m)$ with $W^{1,p}(X;\,\R^m)$, $1\leq p < \infty$, provided suitable conditions are imposed on the problem data $f,g,h_i,l$ and the multipliers $\phi,\psi_i,\theta$ to ensure that all integrals being considered are well defined. For instance, if $p=3$ and $g$ is quadratic in $u$ and $Du$, then $\phi$ should grow no faster than a linear function of $u$ and $Du$, so $\int_X \phi(x,u,Du) g(x,u,Du) \dnx$ is well defined. Finally, following~\cite{Korda2018} it is not difficult to adapt our discussion to variational problems with second-order semilinear PDE constraints and whose objective includes an integral over the boundary. 

In the latter case, however, it does not seem possible to formulate finite-dimensional SDPs unless the unit vector normal to each smooth part of the boundary is either a polynomial, or a rational function with nonvanishing denominator. Fortunately, this is true for many domains encountered in applications, including polyhedral domains. To see the source of the difficulty, observe that when
\begin{equation}
L[u] = \int_X l(x,u,Du) \dnx + \int_{\partial X} l_{\rm b}(x,u) \,\dsigma
\end{equation}
one only needs to redefine each $G_i$ in~\cref{e:inf-sup-sos-3,e:sos-final,e:sos-final-with-slacks} as
\begin{equation}
G_i(x,u) := l_{\rm b}(x,u) + \psi_i(x,u)h_i(x,u) - \theta \cdot \hat{n}(x).
\end{equation}
If $l_{\rm b}$ is a given (nonzero) polynomial, then any nonpolynomial dependence of $\hat{n}$ cannot be absorbed by a judicious choice of $\psi_i$ as we have done in \cref{app:nonpolynomial-normal} for the case $l_{\rm b}=0$. The only tractable situation is when $\hat{n}$ is a rational function with nonvanishing (hence, sign-definite) denominator. In this case, each inequality $G_i(x,u) - c_i \geq 0$ can be strengthened into a weighted SOS constraint after multiplying through by the denominator of $\hat{n}$. Similar considerations hold for the measure-theoretic approach of~\cite{Korda2018} and \cref{s:measure-method}.

Finally, we remark that although the convex relaxation methods of \cref{s:sos-method,s:measure-method} are essentially equivalent for a large class of variational problems, each approach has unique advantages. We have already mentioned that if analytical bounds on $\maxL$ are of interest, then it is more convenient to work with~\eqref{e:sos-final} because any suboptimal choice of Lagrange multipliers produces a valid bound, whereas the convex program over measures in~\cref{e:moments-final} must be solved exactly. On the other hand, studying the relation between the feasible set $\mathcal{M}$ of~\cref{e:moments-final} and the set of occupation and boundary measures may help to identify conditions under which arbitrarily sharp bounds on $\maxL$ are possible. Moreover, for  variational problems more general than~\cref{e:L-star} the formulation of a convex relaxation may be easier if Lagrange multipliers are used instead of measures, or viceversa. The former case includes constructing Lyapunov-like functionals for dynamical systems governed by PDEs, as originally done in~\cites{Valmorbida2014,Valmorbida2015a,Valmorbida2015b,Valmorbida2015c,Ahmadi2014,Ahmadi2015,Ahmadi2016,Ahmadi2017,Ahmadi2018}, as well as parametric problems of the form
\begin{equation}
\adjustlimits \sup_{p \in \mathcal{P}} \inf_{u \in \mathcal{U}(p)} \int_X l(x,u,Du,p) \,\dnx
\end{equation}
where $\mathcal{P}$ is a convex set of parameters and the dependence of $l$ and $\mathcal{U}$ on $p$ is affine. Instead, it is simpler to use occupation and boundary measures when the objective $L[u]$ in~\cref{e:L-star} is not an integral functional, but can still be expressed as a function of finitely many moments that is either convex, or can be bounded from below using convex conditions. One such example is when $L[u]$ is a polynomial of integral functionals. All these extensions to the methods described in this work should be explored further in the future.


\appendix
\section{SDPs with nonpolynomial unit normal vector}
\label{app:nonpolynomial-normal}

When the domain $X$ is a semialgebraic set, the smooth portions of its boundary $\partial X_1,\,\ldots,\,\partial X_s$ are level sets of polynomials. Suppose that each $\partial X_i$ is defined by the equation $S_i(x)=0$, where $S_i$ is a polynomial such that $\|\nabla S_i\| \neq 0$ on $\partial X_i$. Then, the unit normal vector to $\partial X_i$ is given by
\begin{equation}
\hat{n}(x)= \frac{\nabla S_i(x)}{\|\nabla S_i(x)\|}.
\end{equation} 

When $\|\nabla S_i\|$ is not constant but polynomial,~\cref{e:sos-final-with-slacks} can be relaxed into an SDP as described in \cref{ss:sos-sdp} after multiplying the inequality $G_i(x,y)-c_i \geq 0$ by $\|\nabla S_i\|$. 

When $\|\nabla S_i\|$ is not polynomial, one can still formulate and SDP if~\cref{e:sos-final-with-slacks} is modified as follows. Recall that the multiplier $\psi_i$, on which the function $G_i(x,u)$ in~\cref{e:inf-sup-sos-3} depends, is arbitrary. Consequently, without loss of generality we can write 
\begin{equation}
\psi_i(x,u) = \frac{\hat{\psi}_i(x,u)}{\|\nabla S_i(x)\|}
\end{equation}
for some continuous function $\hat{\psi}_i(x,u)$. Then,~\cref{e:inf-sup-sos-3} can be rewritten as
\begin{equation}
\maxL =
\inf_{\substack{u \in W^{1,\infty}(X;\,\R^m)\\u(x) \in Y \text{ a.e.}\\v(x) \in Z \text{ a.e.}}}
\sup_{\substack{\alpha\in\R\\\phi\in C(\Omega)\\\hat{\psi}_i \in C(\Gamma_i)\\\theta \in C^1(\overline{X} \times \overline{Y};\,\R^n)}}
\left\{ \int_X F(x,u,v) \,\dnx + \sum_{i=1}^{s}\int_{\partial X_i} \frac{\hat{G}_i(x,u)}{\|\nabla S_i\|} \,\dsigma \right\},
\end{equation}
where $F$ is as in~\cref{e:F-def} and
$\hat{G}_i(x,u) := \hat{\psi}_i(x,u) h_i(x,u) - \theta(x,u) \cdot \nabla S_i(x)$.
%
Upon estimating the integrals we can replace~\cref{e:sos-final} with
\begin{equation}
\maxL \geq
\sup_{\substack{\alpha\in\R\\\phi\in C(\Omega)\\\hat{\psi}_i \in C(\Gamma_i)\\\theta \in C^1(\overline{X} \times \overline{Y};\,\R^n)}}
\left\{ \abs{X} \inf_{(x,y,z) \in \Omega} F(x,u,v) + \sum_{i=1}^{s} \inf_{(x,y)\in\Gamma_i}\hat{G}_i(x,y) \int_{\partial X_i} \frac{\dsigma}{\|\nabla S_i\|}\right\} =:\hat{\mathbb{D}},
\end{equation}
and~\cref{e:sos-final-with-slacks} with
\vspace*{-3ex}
\begin{equation}
\label{e:modified-sos-nonpolynomial-normal}
\hat{\mathbb{D}} =\!\!\!\!
\sup_{\substack{a,b,c_i\in\R\\\phi\in C(\Omega)\\\hat{\psi}_i \in C(\Gamma_i)\\\theta \in C^1(\overline{X} \times \overline{Y};\,\R^n)}}
\hspace{-15pt}\begin{aligned}
\\
\bigg\{
\abs{X} b + \sum_{i=1}^{s} \int_{\partial X_i} \frac{\dsigma}{\|\nabla S_i\|}c_i:\; 
F(x,y,z) - b &\geq 0 \text{ on } \Omega,\\[-3ex]
\hat{G}_i(x,y) - c_i &\geq 0 \text{ on } \Gamma_i, \;i=1,\,\ldots,\,s
\bigg\}.
\end{aligned}
\end{equation}
The constraints on the righthand side are polynomial inequalities if the problem data and $\phi$, $\hat{\psi}_i$, $\theta$ are polynomials, and can be strengthened into weighted SOS conditions as outlined in \cref{ss:sos-sdp}. Then, lower bounds on $\hat{\mathbb{D}}$ (hence, $\maxL$) can be computed by solving an SDP provided that all integrals $\int_{\partial X_i} \frac{\dsigma}{\|\nabla S_i\|}$ can be computed analytically or approximated numerically.

%

\originalsectionstyle
\bibliography{./reflist.bib}

\begin{bibdiv}
\begin{biblist}

\bib{Ahmadi2018}{article}{
      author={Ahmadi, M.},
      author={Valmorbida, G.},
      author={Gayme, D.},
      author={Papachristodoulou, A.},
       title={{A framework for input-output analysis of wall-bounded shear
  flows}},
        date={2018},
      eprint={http://arxiv.org/abs/1802.04974},
}

\bib{Ahmadi2014}{inproceedings}{
      author={Ahmadi, M.},
      author={Valmorbida, G.},
      author={Papachristodoulou, A.},
       title={{Input-output analysis of distributed parameter systems using
  convex optimization}},
        date={2014},
   booktitle={{Proc. 53\textsuperscript{rd} IEEE Conf. Decis. Control}},
   publisher={IEEE},
     address={Los Angeles, CA},
       pages={4310\ndash 4315},
}

\bib{Ahmadi2015}{inproceedings}{
      author={Ahmadi, M.},
      author={Valmorbida, G.},
      author={Papachristodoulou, A.},
       title={{A convex approach to hydrodynamic analysis}},
        date={2015},
   booktitle={{Proc. 54\textsuperscript{th} IEEE Conf. Decis. Control}},
   publisher={IEEE},
     address={Osaka, Japan},
       pages={7262\ndash 7267},
}

\bib{Ahmadi2016}{article}{
      author={Ahmadi, M.},
      author={Valmorbida, G.},
      author={Papachristodoulou, A.},
       title={{Dissipation inequalities for the analysis of a class of PDEs}},
        date={2016},
     journal={Automatica},
      volume={66},
       pages={163\ndash 171},
         url={http://doi.org/10.1016/j.automatica.2015.12.010},
}

\bib{Ahmadi2017}{article}{
      author={Ahmadi, M.},
      author={Valmorbida, G.},
      author={Papachristodoulou, A.},
       title={{Safety verification for distributed parameter systems using
  barrier functionals}},
        date={2017},
     journal={Syst. Control Lett.},
      volume={108},
       pages={33\ndash 39},
}

\bib{Bertsimas2006}{article}{
      author={Bertsimas, D.},
      author={Caramanis, C.},
       title={{Bounds on linear PDEs via semidefinite optimization}},
        date={2006},
     journal={Math. Program. A},
      volume={108},
      number={1},
       pages={135\ndash 158},
}

\bib{Bogachev2007}{book}{
      author={Bogachev, V.~I.},
       title={{Measure Theory}},
   publisher={Springer Berlin Heidelberg},
        date={2007},
}

\bib{Boyd2004}{book}{
      author={Boyd, S.},
      author={Vandenberghe, L.},
       title={{Convex Optimization}},
   publisher={Cambridge University Press},
        date={2004},
         url={https://doi.org/10.1017/CBO9780511804441},
}

\bib{Korda2018}{article}{
      author={Korda, M.},
      author={Henrion, D.},
      author={Lasserre, J.~B.},
       title={{Moments and convex optimization for analysis and control of
  nonlinear partial differential equations}},
        date={2018},
      eprint={http://arxiv.org/abs/1804.07565},
}

\bib{Lasserre2015}{book}{
      author={Lasserre, J.~B.},
       title={{An introduction to polynomial and semi-algebraic optimization}},
   publisher={Cambridge University Press},
        date={2015},
}

\bib{Laurent2009}{book}{
      author={Laurent, M.},
      editor={Putinar, M.},
      editor={Sullivant, S.},
       title={{Sums of squares, moment matrices and optimization over
  polynomials}},
      series={{Emerging Applications of Algebraic Geometry}. {The IMA Volumes
  in Mathematics and its Applications}},
   publisher={Springer New York},
        date={2009},
      volume={149},
}

\bib{Magron2018}{article}{
      author={Magron, V.},
      author={Prieur, C.},
       title={{Optimal control of linear PDEs using occupation measures and SDP
  relaxations}},
        date={2018},
     journal={IMA J. Math. Control Info.},
       pages={1\ndash 16},
}

\bib{Marx2018}{article}{
      author={Marx, S.},
      author={Weisser, T.},
      author={Henrion, D.},
      author={Lasserre, J.~B.},
       title={{A moment approach for entropy solutions to nonlinear hyperbolic
  PDEs}},
        date={2018},
      eprint={http://arxiv.org/abs/1807.02306},
}

\bib{Putinar1993}{article}{
      author={Putinar, M.},
       title={Positive polynomials on compact semi-algebraic sets},
        date={1993},
     journal={Indiana Univ. Math. J.},
      volume={42},
       pages={969\ndash 984},
}

\bib{Sion1958}{article}{
      author={Sion, M.},
       title={{On general minimax theorems}},
        date={1958},
     journal={Pacific J. Math.},
      volume={8},
      number={1},
       pages={171\ndash 176},
}

\bib{Valmorbida2014}{inproceedings}{
      author={Valmorbida, G.},
      author={Ahmadi, M.},
      author={Papachristodoulou, A.},
       title={{Semi-definite programming and functional inequalities for
  distributed parameter systems}},
        date={2014},
   booktitle={{Proc. 53\textsuperscript{rd} IEEE Conf. Decis. Control}},
   publisher={IEEE},
     address={Los Angeles, CA, USA},
       pages={4304\ndash 4309},
}

\bib{Valmorbida2015b}{inproceedings}{
      author={Valmorbida, G.},
      author={Ahmadi, M.},
      author={Papachristodoulou, A.},
       title={{Convex solutions to integral inequalities in two-dimensional
  domains}},
        date={2015},
   booktitle={{Proc. 54\textsuperscript{th} IEEE Conf. Decis. Control}},
   publisher={IEEE},
     address={Osaka, Japan},
       pages={7268\ndash 7273},
}

\bib{Valmorbida2015c}{article}{
      author={Valmorbida, G.},
      author={Ahmadi, M.},
      author={Papachristodoulou, A.},
       title={{Stability analysis for a class of partial differential equations
  via semidefinite programming}},
        date={2016},
     journal={{IEEE Trans. Automat. Control}},
      volume={61},
      number={6},
       pages={1649\ndash 1654},
}

\bib{Valmorbida2015a}{inproceedings}{
      author={Valmorbida, G.},
      author={Papachristodoulou, A.},
       title={{Introducing INTSOSTOOLS : A SOSTOOLS plug-in for integral
  inequalities}},
        date={2015},
   booktitle={{Proc. 2015 Eur. Control Conf.}},
   publisher={IEEE},
     address={Linz, Austria},
       pages={1231\ndash 1236},
}

\end{biblist}
\end{bibdiv}

\end{document}